\documentclass{amsart}
\usepackage[final]{showkeys}
\usepackage{hyperref}
\usepackage{graphics}
\usepackage{amssymb, amsmath}
\usepackage{amsmath,amsthm}

\newtheorem{theorem}{Theorem}
\newtheorem{lemma}[theorem]{Lemma}
\newtheorem{proposition}[theorem]{Proposition}
\newtheorem{corollary}[theorem]{Corollary}

\def \beq {\begin{eqnarray*}}
\def \eeq {\end{eqnarray*}}
\def \beqn {\begin{eqnarray}}
\def \eeqn {\end{eqnarray}}
\def \bequa {\begin{equation}}
\def \eequa {\end{equation}}

\date{\today}
\title{Optimal estimates
from below for biharmonic Green functions}
\author{Hans-Christoph Grunau}
\address{Fakult\"at f\"ur Mathematik, Otto-von-Guericke-Universit\"at, Postfach 4120,
39016 \linebreak Magdeburg, Germany}
\email{hans-christoph.grunau@ovgu.de}

\author{Fr\'ed\'eric Robert}
\address{Institut \'{E}lie Cartan,
Universit\'{e} Henri Poincar\'{e} Nancy 1,
BP 70239,
54506 Vand{\oe}uvre-l\`{e}s-Nancy Cedex, France}
\email{Frederic.Robert@iecn.u-nancy.fr}

\author{Guido Sweers}
\address{Mathematisches Institut, Universit\"at zu K\"oln,
Weyertal 86-90, 50931 K\"oln, Germany}
\email{gsweers@math.uni-koeln.de}

\subjclass[2000]{Primary 35B51; Secondary 35J40, 35A08}

\begin{document}

\begin{abstract}
Optimal pointwise estimates are derived for the biharmonic Green
function under Dirichlet boundary conditions in arbitrary $C^{4,\gamma}$-smooth domains. 
Maximum principles do not exist for fourth
order elliptic equations and the Green function may change sign. It
prevents using a Harnack inequality as for second order problems and hence
complicates the derivation of optimal estimates. The present estimate is obtained
by an asymptotic analysis. The estimate shows that this Green function is positive
near the singularity and that a possible negative part is small in
the sense that it is bounded by the product of the squared distances to the
boundary.
\end{abstract}

\maketitle

\section{Introduction}
\label{sect:Introduction}

It is well-known that the Green function $G(x,y)$
for second order elliptic equations on bounded domains can be
estimated from above and from below by positive multiples of the
same (positive) function, which is explicitly given in terms of the
distances to the boundary, $d(x),d(y)$, and the distance
$\left|x-y\right|$. See for example \cite{ChungZhao}. The behaviour
of the biharmonic Green function for Dirichlet boundary conditions
should be somehow similar but will
have two crucial distinctions. The singularity of course is of lower
order, namely $n-4$ instead of $n-2$ with $n$ the dimension, but a
more serious distinction is the fact that the biharmonic Green
function is not everywhere positive for most domains. Indeed, the
few known domains with a biharmonic Green function of a fixed
positive sign are balls in arbitrary dimensions, small
perturbations of those balls and of some lima\c{c}ons in $2$ dimensions. See
respectively \cite{Boggio,GrunauSweers1,GrunauRobert} and
\cite{DallacquaSweers3,DallacquaSweers4}.
The results in \cite{DallacquaSweers4} extend and correct  \cite{Hadamard}.

It has been observed numerically on domains with a sign changing
biharmonic Green function that the negative part is rather small
and that it is also not located near the singularity. The aim of
this paper is to give optimal estimates from below. Previous results
concerning smallness of the negative part have been obtained in
\cite{GrunauRobert} for $n\ge 3$ and \cite{DallacquaMeisterSweers}
for $n=2$. With the estimates for the absolute value of that Green
function in \cite{DallacquaSweersJDE} these previously known
estimates are, when $n>4$, as follows:
\begin{equation}
-c \, d(x)^2 d(y)^2 \le G_{\Omega}(x,y) \le c^*  \left|x-y\right|^{4-n}\min\left(1,\frac{d(x)^2 d(y)^2}{\left|x-y\right|^{4}}\right)\label{oldest}
\end{equation}
for all $x,y\in \Omega\subset\mathbb{R}^n$ and where $c,c^*$ are some positive constants only depending on the domain. The distance of $x$ to
the boundary $\partial\Omega$ is defined by
\begin{equation*}d(x):=\inf\left\{\left|x-x^*\right|; x^*\in\partial\Omega\right\},
\end{equation*} 
and $G_{\Omega}$ denotes the said Green function. Let us remind the
reader that the biharmonic Green function $G_{\Omega}$ is such that
\begin{equation*} u(x)=\int_{\Omega}
G_{\Omega}(x,y)~f(y)~dy
\end{equation*}
solves
\begin{equation} \left\{\begin{array}{cr}
\Delta^2 u = f & \text{ in } \Omega, \\[1mm]
u =\left|\nabla u\right| = 0 & \text{on } \partial\Omega .
\end{array}\right. \label{sys}
\end{equation}
In (\ref{oldest}) the dimension is restricted to $n > 4$. As has
been shown in \cite{GrunauRobert} for $n=3,4$ and in
\cite{DallacquaMeisterSweers} for $n=2$, the estimate from below in (\ref{oldest})
holds in all dimensions. An estimate from above has also been proved
for $n\le 4$ but the formula for that estimate is different from
(\ref{oldest}). Those estimates can be found in
\cite{DallacquaSweersJDE} and contain the function $H_{\Omega}$ in
(\ref{eq:twosidedestimatefun}).

The main result is the estimate from below in the following theorem.
For the sake of completeness we include the estimate from above.

\begin{theorem}\label{thm:optimalestimate}
Let  $\Omega\subset \mathbb{R}^n$ $(n\ge 2)$ be a bounded
$C^{4,\gamma}$-smooth domain. Let $G_{\Omega}$ denote the biharmonic
Green function in $\Omega$ for (\ref{sys}). Then there exist
constants $c_1\ge 0$, $c_2 >0$, depending on the domain
$\Omega$, such that we have the following Green function estimate:
\begin{equation}
c_2^{-1}~ H_{\Omega} (x,y)\le G_{\Omega} (x,y)+c_1 d(x)^2d(y)^2\le c_2~ H_{\Omega} (x,y)\label{eq:twosidedestimate}
\end{equation}
for all $ x,y\in\Omega$, where
\begin{equation}
H_{\Omega} (x,y):=\left\{ \begin{array}{l}
\displaystyle   |x-y|^{4-n} \min\left\{
1,\frac{d(x)^2d(y)^2}{|x-y|^4}\right\}\hfill \mbox{\ if\ } n>4,\hspace{-8mm}
\\[4mm]
\displaystyle   \log\left(1+\frac{d(x)^2d(y)^2}{|x-y|^4} \right)\hfill \mbox{\ if\ } n=4,\hspace{-8mm}
\\[4mm]
\displaystyle
d(x)^{2-n/2} d(y)^{2-n/2}  \min\left\{
1,\frac{d(x)^{n/2} d(y)^{n/2} }{|x-y|^n}
\right\}\\[4mm]
\hspace{35mm}\hfill \mbox{\ if\ } n=2,3.\hspace{-8mm}
\end{array}
\right.\label{eq:twosidedestimatefun}
\end{equation}
\end{theorem}

\begin{figure}[hb]
\resizebox{11cm}{!}{\includegraphics{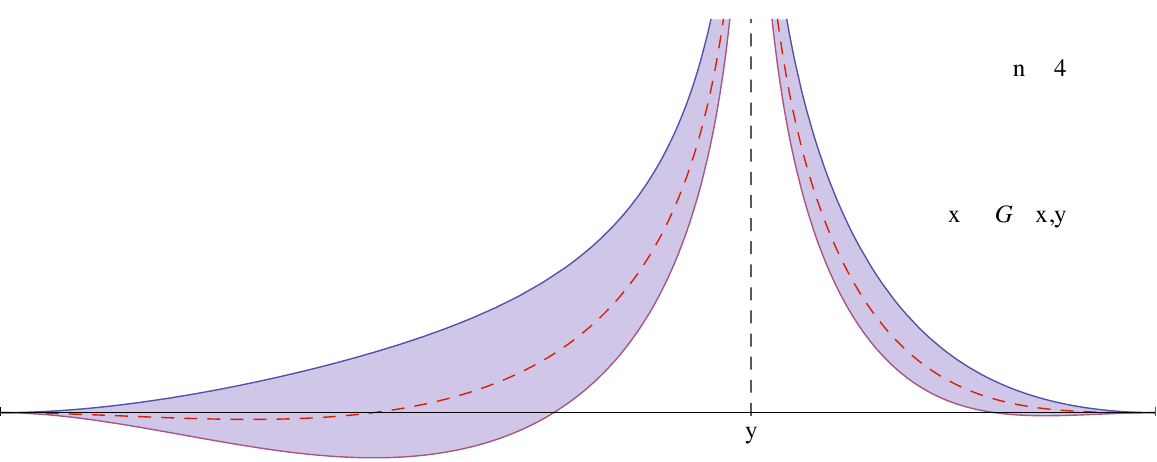}}
\caption{The dashed curve gives the typical behaviour of $G_{\Omega}$; the shaded area describes the band given by $(\ref{eq:twosidedestimate})$.}
\end{figure}

The estimate from above follows from \cite{DallacquaSweersJDE, Krasovskij1,Krasovskij2} so that only
the estimate from below has to be proved here. One should observe
that $G_{\Omega} (x,y)+c_1 d(x)^2d(y)^2\ge 0$ was proved for
some suitable $c_1\ge 0$  in \cite{GrunauRobert}. The preceding
result may be considered as an extension of the estimates in  \cite{GrunauRobert}
showing that close to the pole the positive singular behaviour of the
fundamental solution can also be seen in the Green function.

Due to the different behaviour of the Green function we need to
distinguish between $n\ge 3$ and $n=2$ in proving
Theorem~\ref{thm:optimalestimate}. Finally we should remark that the
lack of a maximum principle not only results in sign changing Green
functions but also complicates the proof of these estimates in the
fourth order case. The proof in the second order case heavily
depends on the Harnack inequality which in turn depends on the
maximum principle.

An interesting consequence of Theorem \ref{thm:optimalestimate} is  a uniform local positivity result.
When $n \ge 3$ this was proved in  \cite{GrunauRobert},
while for  $n=2$ we refer to \cite[Theorem 6.15]{GazzolaGrunauSweers}.
Here the emphasis is on the interplay between Theorem~\ref{thm:optimalestimate}
and the following result. Moreover, we provide a proof for $n=2$ which is
much simpler than the previous one and in the same spirit as for $n\ge 3$.

\begin{theorem}\label{the:unipos}
Let  $\Omega\subset \mathbb{R}^n$ $(n\ge 2)$ be a bounded
$C^{4,\gamma}$-smooth domain. Let $G_{\Omega}$ denote the biharmonic
Green function in $\Omega$ for (\ref{sys}). Then there exist a
constant $r_{\Omega}>0$, such that
\begin{equation}
G_{\Omega}(x,y)>0 \text{ for all } x,y\in\Omega \text{ with } \left|x-y\right|<r_{\Omega}.
\end{equation}
\end{theorem}

\section{Some auxiliary results for $n\ge 3$}
\label{sect:lemmas}

A careful inspection of the proofs in Nehari \cite{Nehari}
and Grunau-Sweers \cite{GrunauSweers3} shows the following
local estimate for the biharmonic Green function from below.

\begin{proposition}\label{NehariTheorem1}
Let $n\ge 3$. Then there exists constants $\delta_n>0$ and $c_3>0$, which
depend only on the dimension $n$, such that the following holds
true.

Assume $\Omega\subset \mathbb{R}^n$ to be a $C^{4,\gamma}$-smooth
bounded domain and let $G_{\Omega}:=G_{\Delta^2,\Omega}$ denote the
Green function for the biharmonic operator under Dirichlet boundary
conditions. If
\begin{equation}
|x-y| \le \delta_n \max \{ d(x), d(y)\},\label{eq:deltanmax}
\end{equation}
then we have
$$
G_{\Omega} (x,y) >
\left\{ \begin{array}{ll}
\displaystyle   c_3~ |x-y|^{4-n}  & \mbox{\ if\ } n>4,
\\[3mm]
\displaystyle   c_3\,\log\left(1+\frac{1}{|x-y|^4} \right)\hfill &\mbox{\ if\ } n=4,
\\[5mm]
\displaystyle
c_3~  d(x)^{1/2} d(y)^{1/2} & \mbox{\ if\ } n=3.
\end{array}
\right.
$$
For the constant $\delta_n$, one may achieve that $\delta_n \ge 0.5$.
\end{proposition}
 In dimension $n=2$ it seems impossible to achieve a global
linear dependence of  the radius of a ball of guaranteed positivity
on the boundary distance, see Lemma \ref{lem:localestimateinterior}
and \cite{Nehari}. In other words, the best result seems $\delta_2 = \delta_2(\Omega)$, which strictly depends on $\Omega$.

Due to Proposition \ref{NehariTheorem1} and using the same constants $\delta_n>0$ as there
we may restrict ourselves in what follows to $x,y$ such that
\begin{equation}\label{eq:condxy}
x,y\in \Omega,x\not=y, |x-y| > \delta_n \max \{ d(x), d(y)\}.
\end{equation}

\begin{lemma}\label{lemma:basicestimate1}
Suppose that $n\ge 3$ and that $\Omega\subset\mathbb{R}^n$ is a bounded
$C^{4,\gamma}$-smooth domain. Then for each $x_0\in\overline{\Omega}$
there exists a radius $r=r_{x_0}>0$ and a constant $C=C_{x_0}>0$
such that for all $x,y\in \Omega_{x_0,r}:=\overline{\Omega}\cap B_r(x_0)$
subject to condition (\ref{eq:condxy}) one has
\begin{equation}\label{eq:boundarybehaviour}
G_{\Omega} (x,y) \ge C |x-y|^{-n} d(x)^2 d(y)^2.
\end{equation}
\end{lemma}

\begin{proof}We only need to discuss $x_0\in\partial\Omega$ since for interior
points $x_0$ one may choose $r=r_{x_0}>0$ so small that condition (\ref{eq:condxy})
becomes void.

We assume by contradiction that there exist
$x_k,y_k\in \Omega_{x_0,1/k}=\overline{\Omega}\cap B_{1/k}(x_0)$ subject to (\ref{eq:condxy})
such that
\begin{equation}\label{eq:contradiction1}
G_{\Omega} (x_k,y_k) < \frac{1}{k} |x_k-y_k|^{-n} d(x_k)^2 d(y_k)^2.
\end{equation}
In particular we have $x_k\to x_0, y_k\to x_0$, $d(x_k),d(y_k)\to 0$, $|x_k-y_k| \to 0$.
Without loss of generality we may assume that $x_0=0$ and that the first unit vector
$\vec{e}_1$ is the exterior unit normal to $\partial\Omega$ at $x_0$.

\begin{figure}[ht]
\hfill\begin{picture}(50,50)(0,0)
\put(0,0){\resizebox{!}{5cm}{\includegraphics{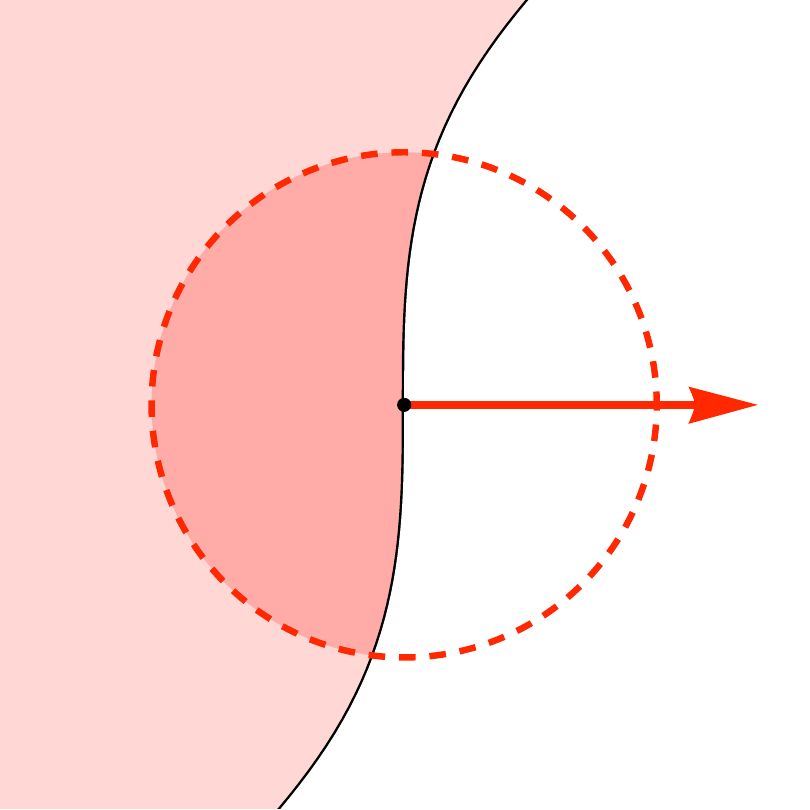}}}
\put(3,45){$\Omega$}
\put(43,21){$\vec{e}_1$}
\put(26,26.5){$x_0$}
\put(12,20){$\Omega_{x_0,1/k}$}
\end{picture}\hfill\mbox{}

\caption{$\Omega$ and subdomain $\Omega_{x_0,1/k}$ for $x_0\in\partial\Omega$.}
\end{figure}
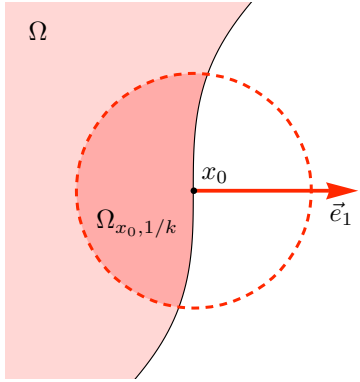

For $k$ large enough, we may define $\tilde{x}_k\in \partial \Omega$ as the closest
boundary point to $x_k$. We introduce the rescaled  biharmonic Green functions
$$
G_k (\xi,\eta) :=|x_k-y_k|^{n-4} G_{\Omega}(\tilde{x}_k+|x_k-y_k|\xi,\tilde{x}_k+|x_k-y_k|\eta)
$$
for
$$
\xi,\eta\in \Omega_k:=\frac{1}{|x_k-y_k|}\left( -\tilde{x}_k+\Omega\right) .
$$
Since $\tilde{x}_k\to 0$, the exterior unit normal at $\partial \Omega$
converges to the first unit vector  and so we conclude that
$$
\Omega_k\to{\mathcal H}:=\{x: x_1<0\} \mbox{\ locally uniformly for \ } k\to \infty.
$$
It was proved in  \cite[Lemma 7]{GrunauRobert}
that locally uniformly in ${\mathcal H}\times{\mathcal H}\setminus\left\{(\xi,\xi);\xi\in\mathcal H\right\}$
$$
G_k(\xi,\eta)\to G_{{\mathcal H}}(\xi,\eta)=
 \frac{1}{4ne_n} |\xi-\eta|^{4-n}
 \int_1^{|\xi^*-\eta|/|\xi-\eta|}
(v^2 -1) v^{1-n}\, dv,
$$
where $\xi^*=(-\xi_1,\xi_2,\ldots,\xi_n)$ and $e_n$ is the $n-$dimensional volume of $B_1(0)\subset\mathbb{R}^n$. We remark that this step
and in particular the required uniqueness proof for $G_{{\mathcal H}}$
were carried out in Grunau-Robert \cite{GrunauRobert} using the assumption $n\ge 3$.
The necessary modifications for $n=2$ are emphasized in the proofs of Lemmata \ref{lem:localestimateinterior} and \ref{lem:localestimateboundary}. Assumption (\ref{eq:contradiction1}) gives
\begin{equation}\label{eq:badassumption}
G_k(\xi_k,\eta_k)=|x_k-y_k|^{n-4}G_{\Omega}(x_k,y_k) <\frac{1}{k} |x_k-y_k|^{-4} d(x_k)^2 d(y_k)^2,
\end{equation}
where
$$
\xi_k =\frac{1}{|x_k-y_k|} (x_k-\tilde{x}_k) ,\quad
\eta_k=\frac{1}{|x_k-y_k|} (y_k-\tilde{x}_k),
$$
$$
|\xi_k|= \frac{d(x_k)}{|x_k-y_k|}\le \frac{1}{\delta_n}, \quad
| \xi_k-\eta_k |=1.
$$
After passing to a further subsequence we find $\xi,\eta\in\overline{\mathcal H}$
with $\xi=\lim_{k\to\infty}\xi_k $, $\eta=\lim_{k\to\infty}\eta_k$.
In view of the local smooth convergence of $G_k$  to the biharmonic
Green function $G_{{\mathcal H}}$ in the half space $\mathcal H$,
Boggio's formula and $|\xi-\eta|=1$
there exists a positive constant $\sigma>0$ such that for $k$ large enough:
\begin{gather*}
G_k(\xi_k,\eta_k)\ge \sigma d(\xi_k)^2 d(\eta_k)^2 = \sigma\left(\frac{d(x_k)}{|x_k-y_k|} \right) ^2
\left( \frac{d(y_k)}{|x_k-y_k|} \right) ^2\\
=\sigma |x_k-y_k|^{-4} d(x_k)^2 d(y_k)^2.
\end{gather*}
This contradicts (\ref{eq:badassumption}) and the proof of the lemma is complete.
\end{proof}

\section{Proof of the main estimate for $n\ge 3$}
\label{sect:mainproof}

Supposing that (\ref{eq:deltanmax}) holds, i.e. $|x-y| \le \delta_n \max \{ d(x), d(y)\}$, one finds that
even $(1-\delta_n)|x-y| \le \delta_n \min \{ d(x), d(y)\}$ and hence
$$
\frac{d(x) d(y)}{|x-y|^2}\ge \frac{1-\delta_n}{\delta_n^2}.
$$
So in that case the estimate from below in Theorem~\ref{thm:optimalestimate}
follows directly from Proposition~\ref{NehariTheorem1}.
Hence, we may assume from now on, again, that
$x,y$ are subject to condition (\ref{eq:condxy}). Applying a compactness argument
to
$$
\overline{\Omega}=\bigcup_{x_0\in \overline{\Omega}} \Omega_{x_0,r_{x_0}/2}
$$
we see that there exist positive numbers $r>0$, $c_4>0$, such that
$|x-y|\le r$ implies that $G_{\Omega} (x,y) \ge c_4 |x-y|^{-n} d(x)^2 d(y)^2$.
If $|x-y|\ge r$  we take from \cite{DallacquaSweersJDE}, cf. also \cite{GazzolaGrunauSweers},
that
$
G_{\Omega} (x,y) \ge -c_5 |x-y|^{-n} d(x)^2 d(y)^2
$
so that
$$
G_{\Omega} (x,y) +2 c_5 |x-y|^{-n} d(x)^2 d(y)^2 \ge c_5 |x-y|^{-n} d(x)^2 d(y)^2.
$$
Since $|x-y|^{-n}\le r^{-n}$ we end up with
$$
G_{\Omega} (x,y) + c_6  d(x)^2 d(y)^2 \ge c_5 |x-y|^{-n} d(x)^2 d(y)^2
$$
and  positive constants $c_5,c_6>0$ in this case. The proof of Theorem~\ref{thm:optimalestimate} for $n\ge 3$ is complete. \hfill $\square$

\section{Auxiliary results for $n=2$}
\label{sect:lemmas_n=2}

\begin{lemma}\label{lem:uniqueGreen}
Let ${\mathcal H}=\{x\in\mathbb{R}^2: x_1<0\}$ and let $\tilde{G}\in C^4(\overline{\mathcal H}
\times\overline{\mathcal H} \setminus \{(x,x):\, x\in \overline{{\mathcal H}}\})$ be a
biharmonic Green function with Dirichlet boundary condition, that is
$$\int_{{\mathcal H}} \tilde{G}(x,\ldotp)\Delta^2\varphi \, dy=\varphi(x) +
\int_{\partial {\mathcal H}}\left(\Delta \tilde{G}(x,\ldotp)
\partial_\nu\varphi-\varphi\partial_\nu\Delta \tilde{G}(x,\ldotp)\right)\, d\sigma$$
for all $\varphi\in C^4_c(\overline{\mathcal H})$ and $x\in\Omega$. Moreover, we assume that $\tilde{G}(x,y)=\tilde{G}(y,x)$
for all $x\neq y$ and that a growth condition holds at infinity:
\begin{equation}
|\tilde{G}(x,y)|\le C\left( 1+|x|^2+|y|^2\right) \left( 1+\left( \log |x|  \right)_+ +\left( \log |y|\right)_+\right).\label{inftybound}
\end{equation}
Then $\tilde{G}$ is uniquely determined and given by Boggio's formula \cite{Boggio}:
$$
\tilde{G}(x,y)=G_{\mathcal H}(x,y)=\frac{1}{8\pi}\left|x-y\right|^2\int_1^{|x^*-y|/|x-y|} \dfrac{v^2-1}{v}\, dv.
$$
with $x^*=(-x_1,x_2)$.
\end{lemma}

\begin{proof} We choose some arbitrary $x\in {\mathcal H}$ and keep it
fixed in what follows. We write
$$
\tilde{G}(x,y)=G_{\mathcal H}(x,y)+H(x,y),
$$
where $H$ is a regular function in $\overline{{\mathcal H}} \times \overline{{\mathcal H}}$
and $\Delta^2_y H(x,\ldotp )\equiv 0$ in ${\mathcal H}$, $H(x,y )=\partial_{y_1} H(x,y )=0$
for $y_1=0$. According to \cite{Duffin,Huber}, or by checking directly, with $y^*=(-y_1,y_2)$,
$$
H^*(x,y):=\left\{ \begin{array}{l}
H(x,y) \hfill \mbox{\ if\ } y_1\le 0,\\[2mm]
\hspace{-2mm}-H(x,y^*)-2y_1\left(\dfrac{\partial H}{\partial y_1} \right) (x,y^*)-
         y_1^2\left(\Delta_y H\right)(x,y^*) \mbox{\ if\ } y_1 >  0,\vspace{-2mm}
\end{array} \right.\vspace{2mm}
$$
satisfies $H^*(x,\ldotp )\in C^4\left( \mathbb{R}^2\right)$ and is biharmonic on $\mathbb{R}^2$. Since $G_{\mathcal H}(x,y)$ satisfies (\ref{inftybound}), so does $H(x,y)$. So we have
$$
| H(x,y) | \le C_x \left( 1+|y|^2\left( \log |y|\right)_+ \right) .
$$
Using local elliptic estimates and their scaling behaviour for biharmonic functions satisfying Dirichlet boundary conditions on $\partial\mathcal H$, that is,
$$
\left\|D^\alpha H(x,\ldotp)\right\|_{L^{\infty}(B_R\cap \mathcal H)}\le\frac{C}{R^{|\alpha|}}\left\|H(x,\ldotp)\right\|_{L^{\infty}(B_{2R}\cap \mathcal H)}
$$
we find that
$$
| \nabla_y H(x,y) | \le C_x \left( 1+|y|\left( \log |y|\right)_+ \right),\qquad
| \nabla^2_y H(x,y) | \le C_x \left( 1+\left( \log |y|\right)_+ \right).
$$
Having these estimates for $H(x,y)$ and these first two derivatives, we have an estimate for $H^*(x,y)$ and may repeat the above arguments to find similar estimates for the derivatives of $H^*(x,y)$ and even
$$
| \nabla^3_y H^*(x,y) | \le C_x \frac{\left( 1+\left( \log |y|\right)_+ \right)}{1+|y|} .
$$
The maximum principle applied to the harmonic function $\nabla_y\Delta_y H^*(x,\ldotp )$ shows that
$$
\| \nabla_y\Delta_y H^*(x,\ldotp ) \|_{C^0(\overline{B_R(0)})}
\le C \frac{\left( 1+\left| \log |R|\right| \right)}{1+|R|}.
$$
Letting $R\to \infty$ yields
$$
\nabla_y\Delta_y H^*(x,\ldotp)\equiv 0, \qquad \Delta_y H^*(x,\ldotp )=a(x)
$$
with a suitable function $a(\ldotp)$.
This shows that any $\nabla^3_y H^*(x,\ldotp)$ is harmonic and, as shown above,
$\nabla^3_y  H^*(x,y )\to 0 $ as $y\to\infty$. Hence, any $\nabla^3_y H^*(x,\ldotp )\equiv 0$
and by Taylor's formula and observing the boundary data we conclude that
$$
H^*(x,y)=b(x)y_1^2
$$
with a suitable function $b(\ldotp )$.
By symmetry $H(x,y)=H(y,x)$ and so $H(x,y)=b(y)x_1^2=b(x)y_1^2=bx_1^2y_1^2,$
where $b\in\mathbb{R}$ is a suitable constant.
Finally, the growth condition leads to $b=0$, $H(x,\ldotp )\equiv0$, and
$\tilde{G}=G_{\mathcal H}$.
\end{proof}

\begin{lemma}[Estimates near the boundary]\label{lem:localestimateboundary}
Suppose that $n=2$ and that $\Omega\subset\mathbb{R}^2$ is a bounded
$C^{4,\gamma}$-smooth domain. Then for each $x_0\in\partial{\Omega}$
there exists a radius $r=r_{x_0}>0$ and a constant $C=C_{x_0}>0$
such that for all $x,y\in \Omega_{x_0,r}:=\overline{\Omega}\cap B_r(x_0)$
 one has
\begin{equation}\label{eq:boundarybehaviour_n=2}
G_{\Omega} (x,y) \ge C d(x) d(y)\min\left\{ 1,\frac{ d(x) d(y)}{|x-y|^{2}} \right\}.
\end{equation}
\end{lemma}

\begin{proof}
We assume by contradiction that there exist
$x_k,y_k\in \Omega_{x_0,1/k}=\overline{\Omega}\cap B_{1/k}(x_0)$
such that
\begin{equation}\label{eq:contradiction11}
G_{\Omega} (x_k,y_k) < \frac{1}{k}d(x_k) d(y_k)\min\left\{ 1,\frac{ d(x_k) d(y_k)}{|x_k-y_k|^{2}} \right\}.
\end{equation}
In particular we have $x_k\to x_0, y_k\to x_0$, $d(x_k),d(y_k)\to 0$, $|x_k-y_k| \to 0$.
Without loss of generality we may assume that $x_0=0$ and that the first unit vector
$\vec{e}_1$ is the exterior unit normal to $\partial\Omega$ at $x_0$.
After possibly passing to a subsequence it is enough to consider one of the
following two cases.

\bigskip
\noindent
{\it First case: $|x_k-y_k|\ge \frac{1}{2} \max\{d(x_k),d(y_k) \}$.}\\
This proof is  as above for Lemma~\ref{lemma:basicestimate1};
only proving the required uniform bounds for the $G_k$ is slightly
more involved. The arguments are sketched below in the second case.
Thanks to Lemma~\ref{lem:uniqueGreen} the convergence proof of
\cite[Lemma 7]{GrunauRobert} can be extended to $n=2$. One should
observe that also the symmetry carries over to the limit.

\bigskip
\noindent
{\it Second case: $|x_k-y_k|< \frac{1}{2} \max\{d(x_k),d(y_k) \}$.}\\
Observe that in this case
$$
d(x_k) <2 d(y_k)< 4d(x_k)
$$
and
$$
|x_k-y_k|<  \min\{d(x_k),d(y_k) \}.
$$
The assumption gives that
\begin{equation}\label{eq:contradiction12}
 G_{\Omega} (x_k,y_k) < \frac{1}{k}d(x_k) d(y_k)
\end{equation}

In this case we rescale differently, however, $\tilde{x}_k\in \partial \Omega$ denotes again the closest
boundary point to $x_k$. We introduce the rescaled  biharmonic Green functions
$$
G_k (\xi,\eta) := d(x_k)^{-2} G_{\Omega}(\tilde{x}_k+d(x_k) \xi,\tilde{x}_k+d(x_k)\eta)
$$
for
$$
\xi,\eta\in \Omega_k:=\frac{1}{d(x_k)}\left( -\tilde{x}_k+\Omega\right) .
$$
Since $\tilde{x}_k\to 0$, the exterior unit normal at $\partial \Omega$
converges to the first unit vector  and so we conclude that
$$
\Omega_k\to{\mathcal H}:=\{x: x_1<0\} \mbox{\ locally uniformly for \ } k\to \infty.
$$
For
$$
\xi_k=\frac{1}{d(x_k)}(x_k-\tilde{x_k}),\qquad \eta_k=\frac{1}{d(x_k)}(y_k-\tilde{x_k})
$$
the assumption (\ref{eq:contradiction12}) transforms into
\begin{equation}\label{eq:contradiction13}
 G_k (\xi_k,\eta_k) < \frac{1}{k}d_k(\xi_k) d_k(\eta_k)< \frac{2}{k},
\end{equation}
where $d_k:=d(\, . \, ,\partial\Omega_k)$. Since $\xi_k,\eta_k$ are bounded  and their boundary distances
are uniformly bounded
from below by $1/2$ we find after passing to a further
subsequence that $\xi_k\to \xi_\infty \in {\mathcal H},\eta_k\to \eta_\infty \in {\mathcal H}$.
We claim that we have local uniform convergence in ${\mathcal H}\times {\mathcal H}$
 (including the diagonal) of $G_k$ to $G_{\mathcal H}$, since we are in
dimension $n=2$. To see this we observe first that  Krasovski\u{\i}-type estimates
(see \cite{Krasovskij1,Krasovskij2} and also \cite[Theorem 4.20]{GazzolaGrunauSweers})
yield at a first instance useful information only for the third derivatives. We have
$$
|\nabla_{\xi,\eta}^3 G_k (\xi,\eta) |\le \frac{C}{|\xi-\eta|} \mbox{\ uniformly in\ }k.
$$
Making use of
$$
\forall \xi \in \partial \Omega_k, \eta \in \Omega_k:\quad  \nabla_\xi \nabla_\eta G_k (\xi,\eta)=0,\quad
 \nabla^2_\eta G_k (\xi,\eta)=0
$$
and of
$$
\forall \xi \in  \Omega_k, \eta \in \partial\Omega_k: \quad \nabla_\xi \nabla_\eta G_k (\xi,\eta)=0,\quad
 \nabla^2_\xi G_k (\xi,\eta)=0
$$
we obtain upon integration that
$$
|\nabla_{\xi,\eta}^2 G_k (\xi,\eta) |\le C(1+(\log |\xi|)_+ + (\log |\eta|)_++ |\log |\xi - \eta||)
\mbox{\ uniformly in\ }k
$$
and further that
$$
|\nabla_{\xi,\eta} G_k (\xi,\eta) |\le C(1+(\log |\xi|)_+ + (\log |\eta|)_+)(1+ |\xi| +  |\eta|)
\mbox{\ uniformly in\ }k;
$$
$$
| G_k (\xi,\eta) |\le C(1+(\log |\xi|)_+ + (\log |\eta|)_+)(1+ |\xi|^2 +  |\eta|^2)
\mbox{\ uniformly in\ }k.
$$
Now, one may proceed further as in \cite[Lemma 7]{GrunauRobert}.
So, (\ref{eq:contradiction13}) yields $G_{\mathcal H}(\xi_\infty,\eta_\infty)\le 0$,
while Boggio's formula shows that $G_{\mathcal H}(\xi,\eta)> 0$ (even if $\xi=\eta$)
since $\xi,\eta\in {\mathcal H}$ are interior points.
\end{proof}

\begin{lemma}[Estimates in the interior]\label{lem:localestimateinterior}
Suppose that $n=2$ and that $\Omega\subset\mathbb{R}^2$ is a bounded
$C^{4,\gamma}$-smooth domain. Then for each $x_0\in {\Omega}$
there exists a radius $r=r_{x_0}>0$ and a constant $C=C_{x_0}>0$
such that for all $x,y\in \Omega_{x_0,r}:=\overline{\Omega}\cap B_r(x_0)$
 one has
\begin{equation}\label{eq:interiorbehaviour_n=2}
G_{\Omega} (x,y) \ge C d(x) d(y)\min\left\{ 1,\frac{ d(x) d(y)}{|x-y|^{2}} \right\}.
\end{equation}
\end{lemma}

\begin{proof}
Since $x_0\in \Omega$, we have that $G_\Omega(x_0,x_0)>0$ (see \cite[p. 115]{Nehari}): since $G_\Omega$
is continuous, there exists $r,c>0$ such that $B_r(x_0)\subset\subset\Omega$ and $G_\Omega(x,y)>c$ for
all $x,y\in B_r(x_0)$. This yields \eqref{eq:interiorbehaviour_n=2} since $\Omega$ is bounded.
\end{proof}

\section{Proof of the main estimate for $n=2$}
\label{sect:mainproof_n=2}

Combining Lemmas \ref{lem:localestimateboundary} and \ref{lem:localestimateinterior} and applying a compactness argument
to
$$
\overline{\Omega}=\bigcup_{x_0\in \overline{\Omega}} \Omega_{x_0,r_{x_0}/2}
$$
we find:

\begin{corollary}\label{lem:2uni}Suppose that $n=2$ and that $\Omega\subset\mathbb{R}^2$ is a bounded
$C^{4,\gamma}$-smooth domain. Then
there exists a radius $r>0$ and a constant $C>0$
such that for all $x_0\in\overline{\Omega}$ and for all $x,y\in \Omega_{x_0,r}:=\overline{\Omega}\cap B_r(x_0)$
 one has
\begin{equation}
G_{\Omega} (x,y) \ge C d(x) d(y)\min\left\{ 1,\frac{ d(x) d(y)}{|x-y|^{2}} \right\}.
\end{equation}
\end{corollary}
This step provides in particular a different and simpler proof for the local
positivity statement from \cite[Theorem 6.15]{GazzolaGrunauSweers}
which was proved first by Dall'Acqua, Meister, and Sweers \cite{DallacquaMeisterSweers}.

\smallskip\noindent{\it Proof of Theorem \ref{thm:optimalestimate}}: If $|x-y|\ge r$  we take from \cite{DallacquaSweersJDE}, see also \cite{GazzolaGrunauSweers},
that
$
G_{\Omega} (x,y) \ge -c_8 |x-y|^{-2} d(x)^2 d(y)^2
$
so that
\begin{gather*}
G_{\Omega} (x,y) +2 c_8 |x-y|^{-2} d(x)^2 d(y)^2 \ge c_8 |x-y|^{-2} d(x)^2 d(y)^2 \\
\ge c_8 d(x)d(y)\min\left\{ 1,\frac{ d(x) d(y)}{|x-y|^{2}} \right\}.
\end{gather*}
Since $|x-y|^{-2}\le r^{-2}$ we end up with
$$
G_{\Omega} (x,y) + c_9  d(x)^2 d(y)^2 \ge c_8 d(x)d(y)\min\left\{ 1,\frac{ d(x) d(y)}{|x-y|^{2}} \right\}.
$$
and  positive constants $c_8,c_9>0$ in this case. The proof of Theorem \ref{thm:optimalestimate} is now complete
also for $n=2$ using Corollary \ref{lem:2uni}. \hfill$\square$

\section{Proof of Theorem \ref{the:unipos}}
This theorem was proved in \cite{GrunauRobert} when $n\geq 3$
and in \cite[Theorem 6.15]{GazzolaGrunauSweers} when $n =2$. The latter
proof is quite involved and based on the extensive use of conformal
maps and explicit Green functions in certain lima\c{c}ons.
We provide here an alternate proof
which uses the same techniques for $n =2$ as for  $n\geq 3$.

\smallskip Case I: $d(x)d(y)\le \left|x-y\right|^2$. For this situation we have
\begin{equation}
H_{\Omega} (x,y)=\left\{ \begin{array}{l}
\displaystyle   |x-y|^{-n} d(x)^2 d(y)^2 \hfill \mbox{\ if\ } n \ne 4,
\\[3mm]
\displaystyle   \log\left(1+\frac{d(x)^2d(y)^2}{|x-y|^4} \right)\hspace{5mm}\hfill \mbox{\ if\ } n=4.
\end{array}
\right.\label{Hbis}
\end{equation}
Then there is $c>0$ such that we find $c_2^{-1} H_{\Omega}(x,y)\ge c_1 d(x)^2d(y)^2$ for $\left|x-y\right|<c$.

Case II: $d(x)d(y) > \left|x-y\right|^2$. Now we have
\begin{equation}
H_{\Omega} (x,y)=\left\{ \begin{array}{l}
\displaystyle   |x-y|^{4-n}  \hfill \mbox{\ if\ } n > 4,
\\[2mm]
\displaystyle   \log\left(1+\frac{d(x)^2d(y)^2}{|x-y|^4} \right)\hspace{5mm}\hfill \mbox{\ if\ } n=4,
\\[4mm]
\displaystyle   d(x)^{2-n/2}d(y)^{2-n/2}\hspace{5mm}\hfill \mbox{\ if\ } n=2,3.
\end{array}
\right.\label{Hbbis}
\end{equation}
Since $d(x)d(y)$ is bounded on $\Omega$ one finds for $n\ge 4$ the existence of $c>0$ such that $c_2^{-1} H_{\Omega}(x,y)\ge c_1 d(x)^2d(y)^2$ for $\left|x-y\right|<c$. For dimension $n=3$ the argument is more subtle. We fix $\delta_n$ as in Proposition \ref{NehariTheorem1}. Taking $d(x)d(y)<\varepsilon^2\frac{1+\delta_n}{2}$ for $\varepsilon>0$ but sufficiently small it follows that $c_2^{-1} H_{\Omega}(x,y)\ge c_1 d(x)^2d(y)^2$. It remains to consider $d(x)d(y)\ge \frac{1+\delta_n}{2}\varepsilon^2$. Assume first that $\left|x-y\right|<\frac{\delta_n}{2} \varepsilon$ and $d(x)<\frac12 \varepsilon$. Then $d(x)d(y)<\frac{1+\delta_n}{2}\varepsilon^2$ and we are in the situation just considered. So we are left with $\left|x-y\right|<\frac{\delta_n}{2}\varepsilon$ and $d(x)\ge \frac12  \varepsilon$. Then we may apply Theorem 1 of \cite{GrunauSweers3}, see also Proposition \ref{NehariTheorem1}, to find that for $\left|x-y\right|<\frac12 \delta_n \varepsilon\le \delta_n \max(d(x),d(y))$ it follows that $G_{\Omega}(x,y)>0$.
For dimension $n=2$ the result follows directly from Corollary \ref{lem:2uni}.

\bigskip

\end{document}